\documentclass[a4paper]{amsart}
\usepackage{graphicx}
\usepackage{amssymb}
\usepackage{enumerate}
\usepackage{lmodern}
\usepackage[utf8]{inputenc}
\usepackage[T1]{fontenc}
\usepackage{hyperref}

\theoremstyle{plain}
  \newtheorem{thm}{Theorem}
  \newtheorem{defn}{Definition}
  \newtheorem{prop}{Proposition}

\theoremstyle{definition}
  \newtheorem{example}{Example}
  \newtheorem{rem}{Remark}

\newcommand{\mc}{\mathcal}

\newcommand{\mf}{\mathfrak}
\newcommand{\on}{\operatorname}

\newcommand{\g}{\mathfrak{g}}
\newcommand{\dd}{\mathfrak{d}}

\newcommand{\la}{\langle}
\newcommand{\ra}{\rangle}
\newcommand{\R}{\mathbb{R}}
\newcommand{\h}{\mf h}

\title{Poisson-Lie T-duality and Courant algebroids}
\author{Pavol \v{S}evera}
\address{Department of Mathematics, Universit\'{e} de Gen\`{e}ve, Geneva, Switzerland}
\email{pavol.severa@gmail.com}
\thanks{Supported in part by  the grant MODFLAT of the European Research Council and the NCCR SwissMAP of the Swiss National Science Foundation.}

\begin{document}
\maketitle

\section{Introduction}
This note explains Poisson-Lie T-duality from the point of view of Courant algebroids. It contains basically nothing new: all the material is already contained in my letters \cite{letters} to Alan Weinstein written in 1998-99, which circulated in the ``Poisson community'' (including, among others, Anton Alekseev, Paul Bressler, Yvette Kosmann-Schwarzbach and Ping Xu) for some time. 

During the 16 years since the letters were written, the basic technical tools (e.g.\ reduction of Courant algebroids \cite{bur}) were rediscovered and works  linking (Abelian) T-duality and Courant algebroids appeared, notably the paper by Cavalcanti and Gualtieri \cite{cg}. I still decided to write my account  and include details missing in \cite{letters}. Perhaps the most important reason is that I introduced exact Courant algebroids while trying to understand Poisson-Lie T-duality, and I believe that this duality, first introduced in \cite{ks}, which generalized the usual Abelian T-duality, is essential for understanding of both Courant algebroids and of the world of T-dualities.

This note summarizes the first four letters of \cite{letters}. In particular, in doesn't deal with differential graded symplectic geometry and its link with Courant algebroids, which is discussed in the remaining letters. While it's certainly relevant for Poisson-Lie T-duality, I decided to exclude it to keep the focus on one thing, and also because I already wrote about it in \cite{some}.

\section{Exact Courant algebroids}
Courant algebroids and Dirac structures were introduced by Liu, Weinstein and Xu in \cite{lwx}.
\begin{defn}
A \emph{Courant algebroid (CA)} is a vector bundle $E\to M$ equipped with a non-degenerate quadratic form $\la\,,\,\ra$, with a bundle map 
$$a:E\to TM$$ (the \emph{anchor map}) and with a $\R$-bilinear map 
$$[\,,\,]:\Gamma(E)\times\Gamma(E)\to\Gamma(E)$$
satisfying
\begin{enumerate}[  A: ]
\item $[s,[t,u]]=[[s,t],u]+[t,[s,u]]$ for any $s,t,u\in\Gamma(E)$ \label{ax1}
\item $a([s,t])=[a(s),a(t)]$ for any $s,t\in\Gamma(E)$ \label{ax2}
\item $[s,ft]=f[s,t]+(a(s)f)t$ for any $s,t\in\Gamma(E)$, $f\in C^\infty(M)$ \label{ax3}
\item $a(s)\la t,u\ra=\la [s,t],u\ra+\la t,[s,u]\ra$ \label{ax4}
\item $[s,s]=a^t\bigl(d\la s,s\ra/2\bigr)$, where \label{ax5}
$a^t:T^*M\to E^*\xrightarrow{\la,\ra} E$
is the transpose of $a$.
\end{enumerate}
A \emph{Dirac structure} in $E$ is a subbundle $L\subset E$ which is Lagrangian w.r.t\ $\la,\ra$ (i.e.\ $L^\perp=L$) such that $\Gamma(L)$ is closed under $[,]$.
\end{defn}
\begin{rem}
This definition from \cite{letters} is somewhat simpler than the (equivalent) original definition of Liu-Wenstein-Xu \cite{lwx}, who used the skew-symmetric part of $[,]$.
\end{rem}
Axiom \ref{ax5} can be replaced by the more innocent-looking
$$
\la s,[t,t]\ra=\la [s,t],t\ra.
$$
Axioms  \ref{ax1}--\ref{ax4} are equivalent to the following: every section $s\in\Gamma(E)$ induces a vector field $Z_s$ on $E$ over $a(s)$, such that the flow of $Z_s$ preserves all the structure; the bracket $[s,s']$ is the Lie derivative of $s'$ under this flow. The map $s\mapsto Z_s$ is $\R$-linear. We have $[Z_s,Z_{s'}]=Z_{[s,s']}$ (as follows from axiom \ref{ax1}). 
\begin{example}[\cite{lwx}]
If $M$ is a point then $E$ is a Lie algebra with invariant non-degenerate quadratic form $\la\,,\,\ra$.
\end{example}
\begin{example}[\cite{lwx}]
If $M$ is a manifold then $E=(T\oplus T^*)M$, with
\begin{subequations}
\begin{align}
\la(u,\alpha),(v,\beta)\ra&=\alpha(v)+\beta(u),\\
a(u,\alpha)&=u,\\
[(u,\alpha),(v,\beta)]&=([u,v],L_u\beta-i_v d\alpha)\label{eq:std}
\end{align}
\end{subequations}
is  the \emph{standard} CA over $M$. In this case $Z_{(u,0)}$ is the natural lift of $u$ to (the natural bundle) $(T\oplus T^*)M$, and $Z_{(0,\alpha)}$ is the vertical vector field with value $-i_vd\alpha$ at $(v,\beta)\in (T\oplus T^*)M$. 
\end{example}

If  $L\subset E$ is a Lagrangian vector subbundle of a CA (i.e.\ if $L^\perp=L$), we can measure the non-involutivity of $L$ (i.e.\ its failure to be a Dirac structure) by
$$\mc F_L:\textstyle\bigwedge^2L\to E/L\cong L^*,\quad \mc F_L(s,t)=[s,t]\text{ mod }L\quad(\forall s,t\in\Gamma(L))$$
where the isomorphism $E/L\cong L^*$ is given by $\la,\ra$, or equivalently by
$$\mc H_L\in\Gamma(\wedge^3L^*),\quad \mc H_L(s,t,u)=\la [s,t],u\ra\quad(\forall s,t,u\in\Gamma(L))$$
(the fact that $\mc F_L$ and $\mc H_L$ are well-defined is readily verified; even though $\mc F_L$ and $\mc H_L$ are really the same object, it will be convenient to have a separate notation).
$L$ is a Dirac structure iff $\mc F_L=0$ (or $\mc H_L=0$).

If $E\to M$ is a CA with anchor map $a$ then
   $a\circ a^t=0$ (as follows from axioms \ref{ax5} and  \ref{ax2}), i.e.
\begin{equation}\label{eq:exact}
0\to T^*M\xrightarrow{a^t} E\xrightarrow{a} TM\to 0
\end{equation}
is a chain complex.
\begin{defn}
A Courant algebroid $E\to M$ is \emph{exact} if \eqref{eq:exact} is an exact sequence. 
\end{defn}

The simplest example of an exact CA is the standard CA; as we shall see below, every exact CA is locally standard.

\begin{example}\label{ex:D/G}
Let $D$ be a Lie group and $G\subset D$ a closed subgroup. Let the Lie algebra $\dd$ of $D$ be equipped with a $Ad$-invariant non-degenerate quadratic form $\la\,,\,\ra_\dd$ such that $\g^\perp=\g$. Then $E=\dd\times D/G$ is an exact Courant algebroid over $D/G$. For constant sections of $E$ the bracket is the Lie bracket on $\dd$ and the anchor is the action of $\dd$ on the homogeneous space $D/G$.

If $\h\subset\dd$ is another Lagrangian Lie subalgebra, i.e.\ if $\h$ is a Dirac structure in $\dd$, then $\h\times D/G\subset \dd\times D/G$ is a Dirac structure.
\end{example}

Exact CAs can be classified in the following way. If $E$ is an exact $CA$ then there is a Lagrangian subbundle $L\subset A$ such that $a|_L:L\to TM$ is an isomorphism (as can be seen by a partition of unity argument). We shall call such a subbundle $L\subset E$ a \emph{connection} in $E$. Equivalently, a connection can be described as a splitting $\sigma:TM\to E$ of the exact sequence \eqref{eq:exact}, such that its image $L$ is Lagrangian.
 Connections  form an affine space over $\Omega^2(M)$ (if $\tau\in\Omega^2(M)$ then $(\tau+\sigma)(v):=\sigma(v)+a^t(i_v\tau)$). 

If $L$ is a connection then one can easily see that 
$$H(u,v,w):=\la[\sigma(u),\sigma(v)],\sigma(w)\ra$$
defines a closed 3-form $H\in\Omega^3(M)$ (if we identify $TM$ with $L$ then $H=\mc H_L$); the 3-form $H$ is called the \emph{curvature} of the connection $L$. If we use  $\sigma\oplus a^t$ to identify $E$ with $TM\oplus T^*M$ then its anchor $a$ and pairing $\la,\ra$ are as in the standard CA, and the bracket is
\begin{equation}\label{eq:exH}
[(u,\alpha),(v,\beta)]=([u,v],L_u\beta-i_v d\alpha+H(u,v,\cdot)).
\end{equation}
On the other hand, for any closed $H$ the bracket \eqref{eq:exH} makes $TM\oplus T^*M$ to an exact Courant algebroid. If we change $\sigma$ by a 2-form $\tau\in\Omega^2(M)$ then $H$ gets replaced by $H+d\tau$. As a result, we have the following theorem:
\begin{thm}[classification of exact CAs]
Exact Courant algebroids over $M$ are classified by $H^3(M,\R)$; exact Courant algebroids with a chosen connection are classified by closed $3$-forms $H$, with the bracket given by \eqref{eq:exH}.
\end{thm}

%

 If $E$ is the exact CA given by \eqref{eq:exH}, and $L\subset E$ a Dirac structure, then on each integral leaf $N\subset M$ of $a(L)\subset TM$ we have a 2-form $\alpha_N$ satisfying $d\alpha_N=H|_N$; the integral leaves and the 2-forms determine $L$ uniquely. 

\section{Exact CAs and 2-dimensional variational problems}\label{sec:var}
Let $\Sigma$ be an oriented surface, $M$ a manifold and $\omega\in\Omega^2(M)$ a 2-form.
Let us consider the functional $S$ on maps $f:\Sigma\to M$ given by
\begin{equation}\label{eq:S}
S[f]=\int_\Sigma f^*\omega.
\end{equation}
 We shall consider more general functionals in Remark \ref{rem:lagr} below; recall, however, that any local functional can be replaced by \eqref{eq:S} if we replace $M$ by an appropriate jet space (the de Donder-Weyl (=multisymplectic) method).

A map $f:\Sigma\to M$ is critical w.r.t.\ $S$ iff
\begin{equation*}
f^*(i_ud\omega)=0
\end{equation*}
for every vector field $u$ on $M$. If $\tau\in\Omega^2(M)$ is closed then $\omega$ and $\omega+\tau$ give equivalent variational problems. More generally, if $H\in\Omega^3_{cl}(M)$ is a closed 3-form, we can consider maps $f:\Sigma\to M$ satisfying
\begin{equation}\label{eq:EL}
f^*(i_u H)=0
\end{equation}
and call them \emph{critical} (or \emph{$H$-critical}).
As $H$ is locally of the form $d\omega$, we can still see this equation as a solution of a variational problem. 

\begin{rem}
From quantum point of view, to make the path integral formally meaningful, one needs to upgrade $H$ to a Cheeger-Simons differential character, or equivalently to a class in the smooth Deligne cohomology \cite{gaw}.
\end{rem}

Let us now consider the exact CA $E\to M$ with connection $L\subset E$ such that its curvature is $H$. (We can set $E=(T\oplus T^*)M$ with the bracket \eqref{eq:exH} and $L=TM$; if $H=d\omega$ we can equivalently take the standard CA and set $L$ to be the graph of $\omega:TM\to T^*M$). For any map $f:\Sigma\to M$ let
$$\tilde Tf:T\Sigma\to L$$
be the lift of the tangent map $Tf:T\Sigma\to TM$ given by $a\circ\tilde Tf=Tf$. The map $\tilde Tf$ can be used to pull back sections of $\bigwedge L^*$ to differential forms on $\Sigma$; this pullback will be denoted by $f^{\tilde*}$.

The Euler-Lagrange equation \eqref{eq:EL} can be rephrased as a `zero-curvature condition'.
\begin{prop}\label{prop:crit}
A map $f:\Sigma\to M$ is critical iff
$$f^{\tilde*}(\mc F_L)=0\ (\in\Omega^2\bigl(\Sigma,f^*(E/L)\bigr)).$$
\end{prop}

The importance of $E$ is that its sections, rather than just vector fields on $M$, can be interpreted as symmetries and give rise to conservation laws.
\begin{thm}[``Noether'']\label{thm:noe}
If $s\in\Gamma(E)$ is such that the flow of $Z_s$ preserves $L$ then for any critical map $f:\Sigma\to M$ the 1-form $f^{\tilde*}\la s,\cdot\ra\in\Omega^1(\Sigma)$ satisfies
$$d\bigl(f^{\tilde*}\la s,\cdot\ra\bigr)=0.$$
\end{thm}
\begin{proof}
We identify $E$ with $(T\oplus T^*)M$ using the connection $L$; the bracket on $E$ is then \eqref{eq:exH}  and $L=TM$. If $s=(u,\alpha)$ then  $f^{\tilde*}\la s,\cdot\ra=f^*\alpha$. The flow of $Z_s$ preserves $L$ iff $$d\alpha+i_uH=0.$$
If $f$ is critical then $f^*(i_uH)=0$. We thus get $d\bigl(f^{\tilde*}\la s,\cdot\ra\bigr)=f^*(d\alpha)=0$.
\end{proof}

The main theme of this paper is the study of symmetries that in place of closed 1-forms give rise to flat connection. The fact that Euler-Lagrange equations can be seen as a zero-curvature condition (Proposition \ref{prop:crit}) will play an important role. To explain it we will need equivariant exact CAs, which we introduce in the following section.


\begin{rem}
If we considered 1-dimensional variational problems instead of 2-dimensional then exact CAs would be replaced by Lie algebroids $A\to M$ which are extensions $0\to\R\to A\to TM\to0$. A splitting of this extension, i.e.\ a connection in $A$, gives rise to a closed 2-form (the curvature of the connection). To make formal sense of the path integral we would rather need a principal $U(1)$-bundle $P\to M$ with a connection; in this case $A=(TP)/U(1)$.

In the case of 2-dimensional problems principal $U(1)$-bundles are replaced by $U(1)$-gerbes (as observed by Brylinski \cite{bry}, reinterpreting Gawędzki's approach via smooth Deligne cohomology \cite{gaw}). Exact CAs are thus closely related to $U(1)$-gerbes.
\end{rem}

\begin{rem}\label{rem:lagr}
If $\Sigma$ is a surface with pseudo-conformal structure, with local light-like coordinates $t_1,t_2$, and $r\in\Gamma({T^*}^{\otimes2}M)$ is a tensor field on $M$, let us consider the standard $\sigma$-model action functional on maps $f:\Sigma\to M$,
$$S_r[f]=\int_\Sigma r\Bigl(\frac{\partial f}{\partial t_1},\frac{\partial f}{\partial t_2}\Bigr)\, dt_1\,dt_2.$$
 Let $E$ be the standard CA and $ R\subset E$ be the graph of $TM\to T^*M$, $v\mapsto r(v,\cdot)$ (then $ R^\perp$ is the graph of $v\mapsto -r(\cdot, v)$). For any $f:\Sigma\to M$ let us lift $Tf:T\Sigma\to TM$ to $\tilde Tf:T\Sigma\to E$ by requiring $\tilde Tf(\partial_{t_1})\in R$ and $\tilde Tf(\partial_{t_2})\in R^\perp$. In this case Noether theorem says:
 
   \emph{If $ R$ is invariant under the flow of $Z_s$ for some $s\in\Gamma(E)$, and if $f:\Sigma\to M$ is critical for the functional $S_r$, then $d\bigl(f^{\tilde*}\la s,\cdot\ra\bigr)=0$. } 

  Proposition \ref{prop:crit} becomes trickier and we don't state it here (see \cite[Section 5]{some}, where it is formulated in terms of differential graded manifolds).
A similar picture can be found for any Lagrangian density depending on the first derivatives of $f$.
\end{rem}

\section{Equivariant Courant algebroids and their reduction}

Let $\mf g$ be a Lie algebra, $E\to M$ a Courant algebroid, and $\rho:\mf g\to\Gamma(E)$ a $[,]$-preserving linear map. The functions $\la\rho(\xi),\rho(\eta)\ra\in C^\infty(M)$, $\xi,\eta\in\g$, are constant (provided $M$ is connected), as
$$0=\rho([\xi,\eta]+[\eta,\xi])=[\rho(\xi),\rho(\eta)]+[\rho(\eta),\rho(\xi)]=d\la\rho(\xi),\rho(\eta)\ra.$$
The resulting (possibly degenerate) pairing $\la\rho(\xi),\rho(\eta)\ra$ on $\mf g$ is $ad$-invariant. This leads us to the following definition.

\begin{defn}
Let $\g$ be a Lie algebra and $\la,\ra_\g$ an invariant symmetric bilinear pairing on $\g$ (possibly degenerate). If $E$ is a CA, a \emph{representation} of $(\g,\,\la,\ra_\g)$ in $E$ is a linear map $\rho:\mf g\to\Gamma(E)$ such that $\rho([\xi,\eta])=[\rho(\xi),\rho(\eta)]$ and $\la\rho(\xi),\rho(\eta)\ra=\la\xi,\eta\ra_\g$.
\end{defn}
A representation of $\g$ in $E$ gives us an action of $\g$ on $E$ by $Z_{\rho(\xi)}$'s. If $G$ is a connected Lie group with the Lie algebra $\g$, and the action of $\g$ on $E$ gives rise to an action of $G$ in $E$, we shall say that $A$ is a \emph{$G$-equivariant CA}.

\begin{rem}
If exact CAs $E\to M$ are seen as approximations of $U(1)$-gerbes over $M$ then the ``gerby'' version of a $G$-equivariant exact CA $E\to M$ is a multiplicative gerbe over $G$ acting on a gerbe over $M$. In this context it's best to replace exact CAs $E\to M$ with principal $\R[2]$-bundles $X\to T[1]M$ in the category of differential graded manifolds. Multiplicative gerbes over $G$ are approximated by central extensions of the group  $T[1]G$ by $\R[2]$, and such extensions are classified by invariant symmetric bilinear forms on $\g$. See \cite[section 3]{some} for some details.
\end{rem}

It is easy to see that $\g$-invariant sections of $E$, orthogonal to the image of $\g$ in $E$, are closed under the bracket $[,]$. This gives us, after we mod out by the sections which are in the kernel of $\la,\ra$, the following theorem.

\begin{thm}[reduction of CAs]
Let $E\to M$ be an $G$-equivariant CA. Suppose that the  action  of $G$ on $M$ is free and proper. For any $x\in M$ let 
$$ (E_{/G})_x=(\rho_x(\g))^\perp/(\rho_x(\g)^\perp\cap\rho_x(\g))=(\rho_x(\g))^\perp/\rho_x(\g'),$$
where $\g'\subset\g$ is the kernel of $\la,\ra_\g$.
After factorization by the action of $G$, $E_{/G}$ becomes a vector bundle over $M/G$. The CA structure on $E$ descends to a CA structure on $E_{/G}$. If $E$ is exact and $\la,\ra_\g=0$ (and thus $\g'=\g$) then $E_{/G}$ is exact.
\end{thm}

\begin{rem}
This reduction procedure was rediscovered and further generalized in \cite{bur}.
\end{rem}

If $M\to M/G$ is a principal $G$-bundle (i.e.\ if $G$ acts freely and properly on $M$), let its \emph{Pontryagin class} be
$$ \bigl[\la F, F\ra_\g\bigr]\in H^4(M,\mathbb R)$$
where $F$ is the curvature of a connection on the principal bundle.


\begin{thm}[classification of equivariant exact CAs]\label{thm:cl-equivCA}
If $G$ acts freely and properly on $M$ then $M$ admits a $G$-equivariant exact CA iff the Pontryagin class of the principal $G$-bundle $M\to M/G$ vanishes. 

There is a natural free and transitive action of the group $H^3(M/G,\mathbb R)$  on the set of isomorphism classes of $G$-equivariant CAs $E\to M$, where the class of a closed 3-form $\gamma\in\Omega^3_{cl}(M/G)$ acts by modifying the bracket on $E\to M$ via $$[s,t]_{new}=[s,t]+a^t\bigl(p^*\gamma(a(s),a(t),\cdot)\bigr).$$
\end{thm}

\begin{proof}
Let us choose a connection $A\in\Omega^1(M,\g)$ on the principal $G$-bundle $M\to M/ G$. If $E\to M$ is a $G$-equivariant exact CA, we can choose a $G$-invariant 
Lagrangian splitting $E\cong (T\oplus T^*)M$ (i.e.\ a connection $L\subset E$), such that 
\begin{equation}\label{rhoA}
\rho(\xi)=\bigl(\xi_M,\frac{1}{2}\la\xi, A\ra_\g\bigr)\in\Gamma\bigl((T\oplus T^*)M\bigr)\quad(\forall\xi\in\g),
\end{equation}
where $\xi_M=a(\rho(\xi))$ is the vector field on $M$ given by $\xi\in\g$.
The bracket on $E\cong (T\oplus T^*)M$ is now given by \eqref{eq:exH} for some $G$-invariant closed 3-form $H\in\Omega^3_{cl}(M)$.

 $G$-invariance of the splitting $E\cong (T\oplus T^*)M$ means that for every vector field $u$ on $M$ and any $\xi\in\g$ the section
 $\bigl[\rho(\xi),(u,0)\bigr]$ of $E$ is a vector field (i.e.\ its 1-form part is zero). Equivalently 
\begin{equation}\label{Hequiv}
i_{\xi_M}H=-\frac{1}{2}\la\xi,d A\ra_\g.
\end{equation}
Equation \eqref{Hequiv} also ensures that $\rho$ is a representation of $\g$ in $E$.

The Chern-Simons 3-form $\mathsf{cs}=\la A,d A\ra_\g+\frac{1}{3}\la[ A, A], A\ra_\g$ satisfies (up to a factor) the same equation
$$i_{\xi_M}\mathsf{cs}=\la\xi,d A\ra_\g,$$
and $d\,\mathsf{cs}=p^*\la F,F\ra_\g$, where $p:M\to M/G$ is the projection and $F$ the curvature of $A$. As a result, the general solution of \eqref{Hequiv} is 
$$H=p^*\theta-\mathsf{cs}/2,\ d\theta=\frac{1}{2}\la F,F\ra_\g.$$
If we change the splitting of $E$ by a 2-form $\tau\in\Omega^2(M/G)$  then $H$ gets replaced by $H+p^*d\tau$, i.e.\ $\theta$ by $\theta+d\tau$. $G$-equivariant CAs over $M$ are thus classified by solutions of $d\theta=\frac{1}{2}\la F,F\ra_\g$ modulo exact 3-forms.
\end{proof}

As an  application of the reduction procedure, 
let us now describe a construction/classification of transitive CAs, i.e. of CAs with surjective anchors.
If $\tilde E\to N$ is a transitive CA with anchor $\tilde a:\tilde E\to TN$
 then $B:=\tilde E/\on{Im} \tilde a^t$ is a transitive Lie algebroid with an invariant inner product on the bundle of vertical Lie algebras.
 \begin{thm}[exact equivariant vs.\ transitive CAs]\label{thm:TCA}
If $M\to N=M/D$ is a principal $D$-bundle  and  $\la,\ra_\dd$ is non-degenerate then the reduction by $D$ gives an equivalence between $D$-equivariant exact CAs $E\to M$ and transitive CAs $\tilde E\to N$ such that $\tilde E/\on{Im} \tilde a^t=(TM)/D$.
\end{thm}

\begin{proof}
If $E\to M$ is a $D$-equivariant CA then the fact that $\tilde E:=E_{/D}$ is transitive and $\tilde E/\on{Im} \tilde a^t=(TM)/D$ follows from the definition of $E_{/D}$.

If $\tilde E\to N$ is a transitive CA such that $\tilde E/\on{Im} \tilde a^t=(TM)/D$ then we can (re)construct a $D$-equivariant exact $E\to M$ as follows.
Let 
$$E:=p^* \tilde E\oplus\dd,$$
with the following structure: the anchor
$$a_{E}:p^* \tilde E\oplus\dd\to TM$$
is the sum of the projection $p^* \tilde E\to p^*B=TM$ and of the natural map $\dd\times M\to TM$, the pairing $\la,\ra_{E}$ is the direct sum of the pairings on $\tilde E$ and on $\dd$, and the bracket is given by
$$[p^*s,p^*t]_{E}=p^*[s,t]_{\tilde E},\ [\xi,\eta]_{E}=[\xi,\eta]_\dd,\ [p^*s,\xi]_{E}=0$$
 for all $s,t\in\Gamma(\tilde E)$, $\xi,\eta\in\dd$.

\end{proof}

\begin{rem}
If $B\to N$ is an arbitrary transitive Lie algebroid  with invariant inner product on its vertical Lie algebras then transitive CAs $\tilde E\to N$ such that $\tilde E/\tilde a^t=B$ exist iff the Pontryagin class of $B$ vanishes; in this case $H^3(N,\mathbb R)$ acts freely and transitively on their isomorphism classes (isomorphisms which are identity on $B$). If $B=(TM)/D$ as above then this result follows from Theorems \ref{thm:cl-equivCA} and \ref{thm:TCA}. For general $B$ it can be proven by a direct calculation; we don't need this result here, so we refer the reader to \cite[no.4]{letters}. This result was rediscovered and extended to regular CAs in \cite{xu}.
\end{rem}

\section{Reduction and curvature}
In this section $\dd$ is a Lie algebra with a non-degenerate invariant symmetric pairing $\la,\ra_\dd$ and $D$ a connected Lie group with Lie algebra $\dd$.

Let $D$ act freely and properly on a manifold $M$ and let $E\to M$ be an equivariant  CA. Let $E_{/D}\to M/D$ be the reduction of $E$; we have $p_D^*E_{/D}=\rho(\dd)^\perp\subset E$, where $p_D:M\to M/D$ is the projection.

Let $ L_D\subset E_{/D}$ be a Lagrangian subbundle. Then
$$L:=p_D^* L_D\subset\rho(\dd)^\perp\subset E$$
is a $D$-invariant Lagrangian subbundle of $\rho(\dd)^\perp$; any $D$-invariant Lagrangian subbundle of $\rho(\dd)^\perp$ is of this form. 

 We  define
$$\mc F_L:\textstyle\bigwedge^2L\to (\rho(\dd)^\perp/L)\cong L^*,\quad \mc F_L(s,t)=[s,t]\text{ mod }L\quad(\forall s,t\in\Gamma(L))$$
and
$$\mc H_L\in\Gamma(\wedge^3L^*),\quad \mc H_L(s,t,u)=\la [s,t],u\ra\quad(\forall s,t,u\in\Gamma(L))$$
(a quick inspection shows that $\mc F_L$ and $\mc H_L$ are well defined). Notice (by using $D$-invariant sections in the definition of $\mc F_L$ and $\mc H_L$) that 
$$\mc F_L=p_D^*\mc F_{ L_D}\text{ and }\mc H_L=p_D^*\mc H_{ L_D}.$$

Let now $G\subset D$ be a Lagrangian subgroup (i.e.\ its Lie algebra $\g$ is a Lagrangian subspace of $\dd$, or equivalently $(\dd,\g)$ is a Manin pair). Let us consider the reduced CA $E_{/G}\to M/G$. We  have a natural identification $E_{/G}\cong\rho(\dd)^\perp/G$ (as $\rho(\g)^\perp=\rho(\g)\oplus\rho(\dd)^\perp$  and thus $\rho(\g)^\perp/\rho(\g)\cong\rho(\dd)^\perp$). Let us define a Lagrangian subbundle $L_G\subset E_{/G}$ to be $L_G=L/G$ (i.e.\ $L=p_G^*L_G$). By using $G$-invariant sections of $L$ we get
\begin{equation}\label{fcfck}
\mc F_L=p_G^*\mc F_{L_G}\text{ and }\mc H_L=p_G^*\mc H_{L_G},
\end{equation}
where $p_G:M\to M/G$ is the projection. As a result we have

\begin{prop}\label{prop:rel}
$p_D^*\mc F_{ L_D}=p_G^*\mc F_{L_G}$ and $ p_D^*\mc H_{ L_D}=p_G^*\mc H_{L_G}$. In particular, $L_G\subset E_{/G}$ is a Dirac structure iff $ L_D\subset E_{/D}$ is a Dirac structure.
\end{prop}

Let us now suppose in addition that $E$ is exact (which implies that $E_{/G}$ is exact) and that its anchor $a:E\to TM$ is injective on $L\subset E$. 
Let $$V:=a(L)\subset TM.$$
 Notice that $$\on{rank}V=\on{rank}L=\dim M-\frac{1}{2}\dim\dd=\dim M/G.$$
The non-involutivity of the distribution $V\subset TM$ is measured by its curvature
$$F_V:\textstyle\bigwedge^2V\to TM/V,\quad F_V(u,v)=[u,v]\text{ mod }V\quad(\forall u,v\in\Gamma(V)).$$
From definitions we get that
\begin{equation}\label{fvfc}
F_V\bigl(a(s),a(t)\bigr)=a\bigl( \mc F_L(s,t)\bigr)\quad\forall s,t\in\Gamma(L).
\end{equation}

\section{Non-Abelian conservation laws and Poisson-Lie T-duality}\label{sec:PL}

Poisson-Lie T-duality is a geometric version of ``non-Abelian Noether theorem'', where a symmetry gives rise to a flat connection instead of a closed 1-form, and also an equivalence between two (or more) variational problems. It was introduced in \cite{ks}. The idea of exact CAs was extracted from this T-duality; the following is a ``coordinate-free'' interpretation of Poisson-Lie T-duality in terms of exact CAs.

Let us use the setup and notation of the previous section: $E\to M$ is a $D$-equivariant exact CA (the action of $D$ on $M$ is free and proper), $L\subset\rho(\dd)^\perp$ is a Lagrangian $D$-invariant subbundle such that the anchor $a$ is injective on $L$, and  $G\subset D$ is a Lagrangian subgroup.

 The Lagrangian subbundle $L_G\subset E_{/G}$ is a connection iff $V:=a(L)\subset TM$ is transverse to the fibers of the projection $M\to M/G$. Supposing this (or removing the points where transversality fails), let 
$$H_G\in\Omega_{cl}^3(M/G)$$
be the curvature of the connection $L_G\subset E_{/G}$ and let $A_G\in\Omega^1(M,\g)$ be the connection on the principal $G$-bundle $p_G:M\to M/G$ whose horizontal bundle is $V\subset TM$.

\begin{thm}[``non-Abelian Noether'']\label{thm:nabn}
If $f:\Sigma\to M/G$ is $H_G$-critical then $f^* A_G$ is a flat connection on the principal $G$-bundle $f^*M\to\Sigma$.
\end{thm}
\begin{proof}
It follows immediately from Proposition \ref{prop:crit} and relations \eqref{fcfck} and \eqref{fvfc}
\end{proof}

Motivated by Proposition \ref{prop:crit}, we shall say that a map $\phi:\Sigma\to M$ is \emph{$L$-critical} if the tangent map $T\phi:T\Sigma\to TM$ can be lifted to a vector bundle map $\tilde T\phi:T\Sigma\to L$ (i.e.\ if the image of $T\phi$ is in $V=a(L)$) and if $\phi^{\tilde*}\mc F_L=0\in\Omega^2(\Sigma,\phi^*L^*)$. Notice that the action of $D$ on $M$ sends $L$-critical maps $\Sigma\to M$ to $L$-critical maps; by the following theorem, such maps are equivalent to $H_G$-critical maps $\Sigma\to M/G$.

\begin{thm}[Poisson-Lie T-duality]\label{thm:PL}
If $\phi:\Sigma\to M$ is $L$-critical then $p_G\circ \phi:\Sigma\to M/G$ is $H_G$-critical. If $\Sigma$ is 1-connected and $f:\Sigma\to M/G$ is $H_G$-critical then there is a $L$-critical map $\phi:\Sigma\to M$ such that $f=p_G\circ \phi$; the map $\phi$ is unique up to the action of $G$.

If $G'\subset D$ is another Lagrangian subgroup then lifting $H_G$-critical maps to $L$-critical and projecting them to $M/G'$ gives us an equivalence between $H_G$-critical maps $\Sigma\to M/G$ and $H_{G'}$-critical maps $\Sigma\to M/G'$.
\end{thm}
\begin{proof}
If $\phi:\Sigma\to M$ is $L$-critical then $H_G$-criticality of $p_G\circ f$ follows from Proposition \ref{prop:crit} and from \eqref{fcfck}. If $\Sigma$ is 1-connected and $f:\Sigma\to M/G$ is $H_G$-critical then by Theorem \ref{thm:nabn} there is a map $\phi:\Sigma\to M$ such that $p_G\circ \phi=f$ and such that the image of $T\phi$ is in $V$ ($\phi$ is unique up to the action of $G$). Relation \eqref{fcfck} and $H_G$-criticality of $f$ then imply that $\phi$ is $L$-critical.
\end{proof}

The name ``Poisson-Lie T-duality'' comes from the case when $\g\cap\g'=0$, i.e.\ when $(\dd,\g,\g')$ is a Manin triple and $G$ and $G'$ a dual pair of Poisson-Lie groups.

\begin{rem}
If we start with a half-dimensional subbundle $ R_D\subset E_{/D}$ in place of $L_D$ (we don't suppose that $ R_D$ is Lagrangian) we obtain a subbundle $R_G\subset E_{/G}$. When we locally trivialize the exact CA $E_{/G}$, we get a $\sigma$-model as described in Remark \ref{rem:lagr}. Theorems \ref{thm:nabn} and \ref{thm:PL} remain valid (after the appropriate modification). It was for this type of models that Poisson-Lie T-duality was originally formulated in \cite{ks} (without using the language of CAs). We don't give details, as it's easier to pass to equivalent $\sigma$-models given by 2-forms / closed 3-forms.

This picture was used in \cite{cg} in the case of Abelian $D$ (without discussing critical maps etc.).
\end{rem}

\section{Dirac structures and boundary conditions (D-branes)}

Let us return to variational problems. Let $M$ be a manifold, $N\subset M$ a submanifold, and let us choose forms $\omega\in\Omega^2(M)$, $\alpha_N\in\Omega^1(N)$. If $\Sigma$ is a surface, let us consider the action functional
$$S[f]=\int_\Sigma f^*\omega+\int_{\partial\Sigma}f^*\alpha_N$$
defined on maps $f:\Sigma\to M$ mapping $\partial\Sigma$ to $N$. The Euler-Lagrange equations for critical $f$'s is now
$$
f^*i_u d\omega=0\text{ on }\Sigma,\quad f^*i_v(\omega|_N+d\alpha_N)=0\text{ on }\partial\Sigma
$$
for every vector fields $u$ on $M$ and $v$ on $N$.

More invariantly and generally, we choose a closed 3-form $H\in\Omega^3_{cl}(M)$ and a 2-form $\beta_N\in\Omega^2(N)$ such that $d\beta_N=H|_N$. Locally then we can find $\omega$ and $\alpha_N$ such that $H=d\omega$ and $\beta_N=\omega|_N+d\alpha_N$. The Euler-Lagrange relations now say
\begin{equation}\label{extr-b}
f^*i_u H=0\text{ on }\Sigma,\quad f^*i_v\beta_N=0\text{ on }\partial\Sigma.
\end{equation}
\begin{rem}
For quantization, to make formal sense of the path integral, the pair $(H,\beta_N)$ should be extended to a relative Cheeger-Simons differential character. This fact was discussed in the case of the WZW-model in \cite{ks2} and in full generality in \cite{zuc}.
\end{rem}

As in Section \ref{sec:var} let $L\subset E$ be the exact CA with connection whose curvature is $H$. Let $C\subset E$ be a Dirac structure. On any leaf $N\subset M$ of the distribution $a(C)\subset TM$ we then have a 2-form $\beta_N\in\Omega^2(N)$ such that $d\beta_N=H|_N$. We can thus use the Dirac structure $C$ to impose a boundary condition; the map $f$ is required to send each component of $\partial\Sigma$ to a leaf $N$, and critical maps are given by the Euler-Lagrange equation \eqref{extr-b}.

Proposition \ref{prop:crit} has now the following form.
\begin{prop}
A map $f:\Sigma\to M$ is critical with the boundary condition given by $C$ iff
$$f^{\tilde*}(\mc F_L)=0\ (\in\Omega^2\bigl(\Sigma,f^*(E/L)\bigr))$$
and 
$$(\tilde Tf)\bigl(T(\partial\Sigma)\bigr)\subset C.$$
\end{prop}

Let us now describe Dirac structures (and thus boundary conditions) compatible with Poisson-Lie T-duality. We shall use the setup of Section \ref{sec:PL}: a free and proper action of $D$ on $M$, a $D$-equivariant exact CA $E\to M$,  a Lagrangian $D$-invariant subbundle $L\subset\rho(\dd)^\perp$ such that $a$ is injective on $L$,  a Lagrangian subgroup $G\subset D$. This data gives us the connection $L_G$ in the exact CA $E_{/G}\to M/G$ with curvature $H_G\in\Omega^3_{cl}(M/G)$.

We can now use Proposition \ref{prop:rel} to describe Dirac structures (or boundary conditions) in $M/G$ compatible with Poisson-Lie T-duality. We start with a Dirac structure $C_D\subset E_{/D}$; by Proposition \ref{prop:rel} (using ``$C$'' in place of ``$L$'') it gives us a Dirac structure $C_G\subset E_{/G}$ and a $D$-invariant subbundle $C\subset\rho(\dd)^\perp\subset E$. 

If $f:\Sigma\to M/G$ is a $H_G$-critical map satisfying the boundary condition given by $C_G$ then its lift $\phi:\Sigma\to M$ (see Theorem \ref{thm:PL}) will satisfy
$$\tilde T\phi(T(\partial\Sigma)\bigr)\subset C$$
and thus, if $G'\subset D$ is another Lagrangian subgroup, the map $p_{G'}\circ\phi:\Sigma\to M/G'$ will be $H_{G'}$-critical (Theorem \ref{thm:PL}) and will satisfy the boundary condition given by $C_{G'}\subset E_{/G'}$.


\begin{thebibliography}{99}
\bibitem{bry}J.-L. Brylinski: Loop Spaces, Characteristic Classes and Geometric Quantization. Prog. Math. 107, Birkhäuser, Boston 1993
\bibitem{bur} H. Bursztyn, G. Cavalcanti, M. Gualtieri: Reduction of Courant algebroids and generalized complex structures, Adv. Math., 211 (2), 2007, 726--765
\bibitem{cg}G. Cavalcanti, M. Gualtieri: Generalized complex geometry and T-duality. In: A Celebration of the Mathematical Legacy of Raoul Bott (CRM Proceedings \& Lecture Notes), American Mathematical Society, 2010, pp. 341--366.
\bibitem{xu} Z. Chen, M. Stienon, P. Xu: On Regular Courant Algebroids. J. Symplectic Geom.
Volume 11, Number 1 (2013), 1--24.

\bibitem{gaw} K. Gawędzki, Topological actions in two-dimensional quantum field theories. In: Nonperturbative Quantum Field Theory
NATO ASI Series Volume 185, 1988, pp 101--141.
\bibitem{ks} C. Klimčík, P. Ševera: Dual non-Abelian T-duality and the Drinfeld double. Phys.Lett. B 351 (1995), 455--462.
\bibitem{ks2} C. Klimčík, P. Ševera: Open strings and D-branes in the WZNW model. Nucl.Phys. B 488 (1997) 653--676.

\bibitem{lwx} Zh.-J. Liu, A. Weinstein, P. Xu: Manin triples for Lie bialgebroids. J. Differential Geom. 45 (1997), no. 3,
547–574
\bibitem{letters} P. Ševera: Letters to Alan Weinstein,  \url{http://sophia.dtp.fmph.uniba.sk/~severa/letters/} (1998--1999).
\bibitem{some} P. Ševera: Some title containing the words “homotopy” and
“symplectic”, e.g. this one. Travaux mathématiques, Volume 16 (2005), 121--137
Journal of Symplectic Geometry 11 (2013), no. 1, 1--24
\bibitem{zuc} R. Zucchini: Relative topological integrals and relative Cheeger-Simons differential characters.  J.Geom.Phys. 46 (2003) 355--393
\end{thebibliography}
\end{document}